\documentclass[12pt,english]{article}
\usepackage[margin=1in]{geometry}
\usepackage{amsmath,amssymb,amsthm,mathrsfs}
\usepackage{bm}
\usepackage{enumerate}

\providecommand{\abs}[1]{\lvert#1\rvert}

\usepackage{marvosym}
\newtheorem{lemma}{Lemma}
\newtheorem{theorem}{Theorem}
\newtheorem{proposition}{Proposition}

\newtheorem{assumption}{Assumption}
\newtheorem{remark}{Remark}

\begin{document}
\baselineskip .275in

\title{Ergodic inventory control with diffusion demand and general ordering costs}

\author{Bo Wei \quad \quad Dacheng Yao\textsuperscript{\Letter}}
\date{}
\maketitle

\begin{abstract}
\noindent  In this work, we consider a continuous-time inventory system where the demand process follows an inventory-dependent diffusion process. The ordering cost of each order depends on the order quantity and is given by a general function, which is not even necessarily continuous and monotone. By applying a lower bound approach together with a comparison theorem, we show the global optimality of an $(s,S)$ policy for this ergodic inventory control problem.

\end{abstract}

\section{Introduction}
\label{sec:Introduction}

This paper is a sequel to \cite{HeETAL17}, which investigates a continuous-time inventory system with a Brownian demand process and a quantity-dependent setup cost. In this setting, an $(s,S)$ replenishment policy turns out to be optimal under the average cost criterion. In \cite{HeETAL17}, the setup cost function is only required to be a nonnegative, bounded, and lower semicontinuous function of the order quantity. It is necessary to consider such a general ordering cost structure, because in practice, expenses arising from administration and transportation may not be continuous in the order quantity.
Furthermore, general ordering cost structure was studied by \cite{PereraETAL17,PereraETAL18} in inventory models with deterministic demand and renewal demand, respectively.

In this work, we establish the global optimality of  an $ (s, S) $ policy for ergodic inventory control with an inventory-dependent diffusion demand process under a general ordering cost structure. One may refer to \cite{CadenillasETAL10,BaronETAL11} for state-dependent inventory models and their applications. Ergodic inventory control with a diffusion demand process has been studied in two recent papers by Helmes et al.\ \cite{HelmesETAL17,HelmesETAL18}. More specifically, an $ (s,S) $ policy is proved to be optimal in a subset of admissible policies in \cite{HelmesETAL17}, in which the authors assume the ordering cost is continuous with respect to the order quantity.
In \cite{HelmesETAL18}, the authors proposed a weak convergence approach, which allow them to further  show the global optimality of an $ (s,S) $ policy among all admissible policies. Our work complements their papers by allowing for a more general ordering cost function that may have discontinuities.

The main results in this paper provide a rigorous justification for the following intuitive interpretation of the optimality of $ (s,S) $ policies for ergodic inventory control:
If the demand process has almost sure continuous sample paths, the inventory administrator is allowed to replenish inventory at any level as she wants. Moreover, if the demand process is also Markovian, the distribution of future demand can be determined based on the current state (inventory level).
In this case, an $(s,S)$ policy would be optimal to minimize the average cost, even a general ordering cost function is involved. Such a simple optimal policy stands in stark contrast with optimal ergodic control in discrete-time inventory models: the inventory administrator is only allowed to replenish the inventory at the start of each period,
the reorder level would be different from period to period. Thus, if the setup cost function is not a constant, this dynamic optimization problem would be generally difficult to tackle (see, e.g., \cite{ChaoZipkin08,Caliskan-DemiragETAL12,ZhangCetinkaya17}).

The remainder of this paper is organized as follows. The diffusion inventory model is introduced and the main results are presented in Section \ref{sec:Model-Results}. An $(s,S)$ policy is selected and is proven to be the best one in a subset of admissible policies by a lower bound theorem in Section \ref{sec:subset}. A comparison theorem is provided to establish the global optimality of the $(s,S)$ policy among all admissible policies in Section \ref{sec:comparison}. Finally, Section \ref{sec:concluding} concludes the study.

\section{Problem Formulation and Main Results}
\label{sec:Model-Results}
\subsection{Diffusion Inventory Model}
\label{sec:Model}

Consider a single-item inventory model, where the inventory level process is governed by
\begin{equation}
\label{eq:Z}
Z(t) = x - D(t)+ Q(t), \quad t\geq0,
\end{equation}
where $Z(0-)=x$ denotes the initial inventory level, $D(t)$ and $Q(t)$ represent the cumulative demand process and cumulative order quantity up to time  $t$, respectively.
The inventory-dependent demand process $\{D(t)\}_{t\geq0}$ is represented as
\[
D(t) = \int_{0}^{t} \mu(Z(s))\,\mathrm{d}s + \int_{0}^{t} \sigma(Z(s))\,\mathrm{d}B(s),
\]
where $\{B(t)\}_{t\geq0}$ denotes a standard Brownian motion on $ (\Omega, \mathcal{F}, \mathbb{P}; \mathcal{F}_t, t\geq0) $. We assume that the drift coefficient $\mu(\cdot)$ and the diffusion coefficient $\sigma(\cdot)$ satisfy the following conditions.
\begin{assumption}\label{assumption:coefficients}
\begin{enumerate}[$(a)$]
\item $\mu(\cdot)$ is continuously differentiable, nondecreasing with
$\underline{\mu}:=\lim_{z\to-\infty}\mu(z)>0$ and $\bar{\mu}:=\lim_{z\to\infty}\mu(z)<\infty$.
\item $\sigma(\cdot)$ is continuous, and $\sigma(\cdot)\in[\underline{\sigma},\bar{\sigma}]$, where $\underline{\sigma},\bar{\sigma}>0$ are two finite constants.
\end{enumerate}
\end{assumption}
Without any replenishment, the inventory level process turns out to be a diffusion process $\{X(t)\}_{t\geq0}$ given by
\begin{equation}
\label{eq:inventory-no-order}
X(t) = x - \int_{0}^{t} \mu(X(s))\,\mathrm{d}s - \int_{0}^{t} \sigma(X(s))\,\mathrm{d}B(s).
\end{equation}
For later use, we denote the scale function of $ X $ by
\[
\mathcal{S}(x) = \int_{a}^{x} \exp \Big( \int_{a}^{y} \frac{2\mu(v)}{\sigma^{2}(v)}\,\mathrm{d}v \Big)\, \mathrm{d}y \quad \mbox{for $ x \in \mathbb{R} $},
\]
where $a$ is an arbitrary real number, and the speed measure of $X$ by
\[
\mathcal{M}(\mathrm{d}x) = \frac{1}{\sigma^{2}(x)} \exp \Big(- \int_{a}^{x} \frac{2\mu(v)}{\sigma^{2}(v)}\,\mathrm{d}v \Big) \mathrm{d}x.
\]

We represent the ordering policy by a cumulative order process $Q =  \{Q(t)\}_{t\geq0} $, which is called \emph{admissible} if it satisfies the three conditions as follows: (i) $Q(t)$ is nonnegative for all $ t \geq 0 $;
(ii) The sample paths of $Q$ are nondecreasing and right-continuous with left limits (RCLL); (iii) $Q$ is adapted.

In this work, the ordering cost function $c(\cdot)$ is assumed to satisfy the following conditions.
\begin{assumption}{\label{assumption:c}}
The function $c:\mathbb{R}_+\to\mathbb{R}_+$ is subadditive\footnote{A function $c:\mathbb{R}_+\to\mathbb{R}_+$ is subadditive if $c(\xi_1+\xi_2)\leq c(\xi_1)+c(\xi_2)$ for $\xi_i\geq0$, $i=1,2$.} and lower semicontinuous\footnote{A function $c:\mathbb{R}_+\to\mathbb{R}_+$ is lower semicontinuous if $c(\xi')\leq \liminf_{\xi\to\xi'}c(\xi)$ for each $\xi'>0$.} with
$c(0)=0$ and $c(0+):=\lim_{\xi\downarrow 0} c(\xi)>0$.
\end{assumption}
The ordering cost function satisfies the condition above is very general, and it is not even necessarily continuous (cf. \cite{HelmesETAL17,HelmesStockbridgeZhu2018} for continuous ordering cost) and monotone. In particular, it includes the classical linear cost (cf. \cite{Scarf60,Iglehart63}), all unit quantity discount cost (cf. \cite{AltintasETAL08}), incremental quantity discounted cost (cf. \cite{Porteus71,YaoETAL15}), and quantity-dependent setup cost (cf. \cite{ChaoZipkin08,Caliskan-DemiragETAL12})  as special cases.

Since  $c(0+)=\lim_{\xi\downarrow 0} c(\xi)>0$, we only need to consider impulse control policies,
which can be specified by $\{(\tau_n,\xi_n):n=0,1,2,\cdots\}$ with that $\tau_n$ and $\xi_n$ denote the time and the amount of $n$th order, respectively. For convenience, we assume that $\tau_0=0$ and $\xi_0\geq0$, i.e., no order is placed when $\xi_0=0$.
Then, an admissible policy $Q$ can be denoted as
$Q(t)=\sum_{n=0}^{N(t)}\xi_n$,
where $N(t)=\max\{n\geq0:\tau_n\leq t\}$. We define $\Phi$ as the set including all such admissible policies.

In addition, let $h(z)$ represent the holding and shortage cost rate for inventory level $z\in\mathbb{R}$.
\begin{assumption}
\label{assumption:h}
The function $h:\mathbb{R}\to\mathbb{R}_+$ is polynomially bounded, convex, continuously differentiable except at $z=0$ with $h(0)=0$. Further,  $h'(z)>0$ if $z>0$, and $h'(z)<0$ if $z<0$.
\end{assumption}
\begin{remark}
\label{rem:h}
$(a)$ The boundedness of the coefficient in Assumption \ref{assumption:coefficients} and polynomial boundedness of $h$ in Assumption \ref{assumption:h} imply that
\[
\int_x^{\infty} h(y)\,\mathcal{M}(\mathrm{d}y)<\infty.
\]
$(b)$ Assumption \ref{assumption:h} implies $\lim_{\abs{x}\to\infty}h(x)=\infty$.
\end{remark}

We need to find an admissible policy $Q \in \Phi$ to minimize the following long-run average cost:
\begin{align}
\label{eq:problem}
\mathcal{C}(x,Q)=\limsup_{t\to\infty}\frac{1}{t}\mathbb{E}_x\Big[\int_0^t h(Z(u))\,\mathrm{d}u+\sum_{n=0}^{N(t)}c(\xi_n)\Big],
\end{align}
where $\mathbb{E}_x[\cdot]:=\mathbb{E}_x[\cdot|Z(0-)=x]$.

\subsection{Main Results}

Under an $(s, S)$ policy, a cycle is defined as the duration from $S$ to $s$.
Then, the controlled process $Z$ can be regarded as a regenerative process. Using the regenerative process theory, we have
\[
\alpha(s, S):=\mathcal{C}(S,(s,S))=\frac{\mathbb{E}_S[\int_0^{\tau_S^s}h(Z(u))\,\mathrm{d}u]+c(S-s)}{\mathbb{E}_S[\tau_S^s]},
\]
where $\tau_S^s$ is the duration time of one cycle.
Under Assumptions \ref{assumption:coefficients} and \ref{assumption:h}, we have
\begin{align}
\label{eq:E-h-tau}
\mathbb{E}_S\Big[\int_0^{\tau_S^s}h(Z(u))\,\mathrm{d}u\Big]=2\int_s^S \int_x^{\infty}h(y) \,\mathcal{M}(\mathrm{d}y) \,\mathrm{d}\mathcal{S}(x) \text{ and }
\mathbb{E}_S\big[\tau_S^s\big]=2\int_{s}^S \int_x^{\infty} \mathcal{M}(\mathrm{d}y)\,\mathrm{d}\mathcal{S}(x);
\end{align}
see Proposition 2.6 in \cite{HelmesETAL17}.
Therefore,
\begin{align}
\label{eq:alpha}
\alpha(s, S)=\frac{2\int_s^S \int_x^{\infty} h(y)\,\mathcal{M}(\mathrm{d}y)\,\mathrm{d}\mathcal{S}(x)+c(S-s)}{2\int_s^S\int_{x}^{\infty} \,\mathcal{M}(\mathrm{d}y)\,\mathrm{d}\mathcal{S}(x)}.
\end{align}
Under $(s,S)$ policy, for any initial state $x\in\mathbb{R}$, level $S$ can be reached in finite expected time due to strictly positive demand drift.
Actually, $\alpha(s, S)$ is the average cost which is independent of the initial state $x\in\mathbb{R}$, i.e., $\alpha(s, S)=\mathcal{C}(x,(s,S))$ for any $x\in\mathbb{R}$. In the following lemma, we claim the existence of the best $(s,S)$ policy in minimizing $\alpha(s,S)$.
\begin{lemma}
\label{lem:sS}
Under Assumptions \ref{assumption:coefficients}-\ref{assumption:h}, there exists a finite pair $(s^{\star},S^{\star})$ with $s^{\star}<S^{\star}$ satisfying
\begin{equation}
\label{eq:sS-star}
(s^{\star},S^{\star})=\arg\inf_{s<S} \alpha(s,S).
\end{equation}
\end{lemma}

Our main results are as follows.
\begin{theorem}
\label{thm:mainresults}
Suppose Assumptions \ref{assumption:coefficients}-\ref{assumption:h} hold.
The $(s^{\star},S^{\star})$ policy given by \eqref{eq:sS-star} is optimal for the ergodic inventory control problem \eqref{eq:problem}
and $\alpha^{\star}:=\alpha(s^{\star},S^{\star})$ is the optimal cost, where $\alpha(s,S)$ is defined in \eqref{eq:alpha}.
\end{theorem}

Theorem \ref{thm:mainresults} will be proven by two steps. First, in Section \ref{sec:subset}, by a lower bound theorem, we show that the $(s^{\star},S^{\star})$ policy is the best one in a subset of $\Phi$. Then in Section \ref{sec:comparison}, we show its global optimality in $\Phi$ by a comparison theorem.

\section{Optimality of the (s,S) Policy in A Subset}
\label{sec:subset}

In this section, by a lower bound theorem, we show that the $(s^{\star},S^{\star})$ policy is the best one in a subset of admissible policies.
Specifically, in Proposition \ref{prop:LBT}, we show that if some function $f$ with certain properties and a constant $\alpha$ satisfy the lower bound conditions  \eqref{eq:LB-condition-1}-\eqref{eq:LB-condition-3},
then the cost under any policy in a subset $\Phi_f$ is larger than $\alpha$. We construct a function $V$ and in Proposition \ref{prop:optimal} check that $f=V$ and $\alpha=\alpha^{\star}$ satisfy all lower bound conditions.
Thus, $\alpha^{\star}=\alpha(s^{\star},S^{\star})$ is a lower bound of the cost under any $Q\in\Phi_V$, i.e., $(s^{\star},S^{\star})$ policy is optimal in $\Phi_V$. Finally, in Proposition \ref{prop:bar-V}, we show that $\Phi_V$ is large enough to include a class of admissible policies with order-up-bounds.

Let $\mathscr{A} f(z)=\frac{1}{2}\sigma^2(z) f''(z)-\mu(z)f'(z)$.
The following proposition provides a lower bound theorem.
See Proposition 2 in \cite{HeETAL17} for a similar proof.
\begin{proposition}[Lower Bound Theorem]
\label{prop:LBT}
Suppose Assumption \ref{assumption:h} holds.
Let $f$ be a real-value function with absolutely continuous $f'$, and let $\alpha$ be a positive number. If
\begin{align}
\label{eq:LB-condition-1}
\mathscr{A} f(z)+h(z)\geq \alpha\quad \text{for any $z\in\mathbb{R}$ when $f''(z)$ exists},
\end{align}
with
\begin{align}
&f(z_2)-f(z_1)\geq -c(z_2-z_1)\quad \text{for any $z_2>z_1$},\quad \text{and}\label{eq:LB-condition-2}\\
&\abs{f'(z)}<a_0\quad \text{for all $z<0$ and some positive number $a_0$},\label{eq:LB-condition-3}
\end{align}
then we have  $\mathcal{C}(x,Q)\geq \alpha$ for each $Q\in\Phi_f$ and each $x\in\mathbb{R}$,
where $\Phi_f\subset \Phi$ consists of those policies $Q$ such that their resulting inventory process $Z$ satisfying
\begin{align}
 (i) &\quad \mathbb{E}_x\Big[\int_0^t \big(f'(Z(s))\sigma(Z(s))\big)^2\,\mathrm{d}s\Big]<\infty \text{ for $t\geq0$};\label{eq:condition1}\\
 (ii)&\quad \mathbb{E}_x[\abs{f(Z(t))}]<\infty \text{ for $t\geq0$}; \text{ and}\label{eq:condition2}\\
(iii) &\quad \lim_{t\to\infty}\frac{1}{t} \mathbb{E}_x[\abs{f(Z(t))1_{\{Z(t)\geq0\}}}]=0.\label{eq:condition3}
\end{align}
\end{proposition}

We next construct a function, embodied by $V$, which together with $\alpha^{\star}=\alpha(s^{\star},S^{\star})$, satisfies all conditions in Proposition \ref{prop:LBT}.
Define
\begin{align*}
g(z)=2 \mathcal{S}'(z)\int_{z}^{\infty} h(u)\,\mathcal{M}(\mathrm{d}u)\quad \text{and}\quad \ell(z)= 2\mathcal{S}'(z)\int_{z}^{\infty}\,\mathcal{M}(\mathrm{d}u),
\end{align*}
Note that $g$ and $\ell$ satisfy
\begin{equation}
\label{eq:g-l-equations}
\frac{\sigma^2(z)}{2} g'(z)-\mu(z)g(z)+h(z)=0\quad\text{and}\quad \frac{\sigma^2(z)}{2} \ell'(z)-\mu(z)\ell(z)+1=0.
\end{equation}
\begin{lemma}
\label{lem:underline-s}
If Assumptions \ref{assumption:coefficients}-\ref{assumption:h} hold,  then there is an $\underline{s}$ with $\underline{s}\leq s^{\star}$ such that
\begin{align}
&\underline{\alpha}(s,S):= \frac{\int_s^S g(y\vee\underline{s})\,\mathrm{d}y+c(S-s)}{\int_s^S \ell(y\vee\underline{s})\,\mathrm{d}y}\geq \alpha^{\star}\quad \text{and} \label{eq:underline-s1}\\
&g'(z)-\alpha^{\star}\ell'(z)<0\quad \text{for all $z\leq \underline{s}$}.\label{eq:underline-s2}
\end{align}
\end{lemma}

Now we are ready to construct the function $V$ as follows.
\begin{align}
\label{eq:V}
V(z)&=\int_{\underline{s}}^z g(\max(y,\underline{s}))\,\mathrm{d}y-\alpha^{\star} \int_{\underline{s}}^z \ell(\max(y,\underline{s}))\,\mathrm{d}y
=\begin{cases}
\int_{\underline{s}}^z g(y)\,\mathrm{d}y-\alpha^{\star} \int_{\underline{s}}^z \ell(y)\,\mathrm{d}y& \text{for $z\geq \underline{s}$},\\
[g(\underline{s})-\alpha^{\star} \ell(\underline{s})](z-\underline{s})& \text{for $z<\underline{s}$}.
\end{cases}
\end{align}
Next, we show that $V$ and $\alpha^{\star}$ satisfy conditions \eqref{eq:LB-condition-1}-\eqref{eq:LB-condition-3}, and then Proposition \ref{prop:LBT} implies that $\alpha^{\star}=\alpha(s^{\star},S^{\star})\leq \mathcal{C}(x,Q)$ for $Q\in\Phi_{V}$, i.e., $(s^{\star},S^{\star})$ policy is optimal in $\Phi_{V}$.

\begin{proposition}
\label{prop:optimal}
If Assumptions \ref{assumption:coefficients}-\ref{assumption:h} hold, we have that $\mathcal{C}(x,Q)\geq \alpha^{\star}$ holds for each $Q\in\Phi_V$ and $x\in\mathbb{R}$.
\end{proposition}

\begin{proof}[Proof of Proposition \ref{prop:optimal}]
We will claim that $(V,\alpha^{\star})$ satisfies all conditions of Proposition \ref{prop:LBT}.
First, $V$ defined in \eqref{eq:V} is continuously differentiable in whole $\mathbb{R}$ and $f''$ exists except at $\underline{s}$, thus $V$ is continuously differentiable with absolutely continuous $V'$.

We next verify \eqref{eq:LB-condition-1}. From \eqref{eq:g-l-equations} and \eqref{eq:V}, we have
that for $z\geq \underline{s}$, $\mathscr{A} V(z)+h(z)=\alpha^{\star}$ holds. Further,  for $z<\underline{s}$, we have
\begin{align*}
\mathscr{A} V(z)+h(z)
&=-\mu(z)\big[g(\underline{s})-\alpha^{\star}\ell(\underline{s})\big]+h(z)\\
&\geq\frac{\sigma^2(z)}{2}\big[g'(z)-\alpha^{\star}\ell'(z)\big]-\mu(z)\big[g(z)-\alpha^{\star}\ell(z)\big]+h(z)\\
&=\alpha^{\star},
\end{align*}
where the inequaly holds due to  \eqref{eq:underline-s2}, and the last equality is derived from \eqref{eq:g-l-equations}.

Now we check \eqref{eq:LB-condition-2}.  It follows from \eqref{eq:underline-s1} that for $z_1<z_2$,
\begin{align*}
V(z_2)-V(z_1)
=\int_{z_1}^{z_2} g(\max(y,\underline{s}))\,\mathrm{d}y-\alpha^{\star} \int_{z_1}^{z_2} \ell(\max(y,\underline{s}))\,\mathrm{d}y
\geq -c(S-s).
\end{align*}

Finally, we prove \eqref{eq:LB-condition-3}. It follows from \eqref{eq:V} that
for $z<0$,
\begin{align*}
\abs{V'(z)}
<\max\{g(\underline{s})-\alpha^{\star}\ell(\underline{s}) , \max_{z\in[\underline{s},0]}(g(z)-\alpha^{\star}\ell(z))\}+1.
\end{align*}
\end{proof}

To the end, we study how large is the subset $\Phi_V$. We define another subset of admissible policies as follows and then show that it is included in $\Phi_V$. For $j\in\mathbb{N}$, let
\[
\Phi(j)=\{Q\in \Phi: Z(\tau_n)\leq j \quad \text{for all $n\geq0$}\},
\]
i.e., under $Q\in\Phi(j)$, the inventory level after ordering at any ordering time does not exceed level $j$.
Let
\[
\bar{\Phi}=\cup_{j=1}^{\infty} \Phi(j).
\]
We will show that $\bar{\Phi}\subseteq\Phi_{V}$. To achieve that, we first provide some properties of $V$ which will be used in proving Proposition \ref{prop:bar-V}.
\begin{lemma}
\label{lem:V}
If Assumptions \ref{assumption:coefficients}-\ref{assumption:h} hold, then there exist a $\bar{z}$ with $0<\bar{z}<\infty$ such that
\begin{equation}
\label{eq:bar-z}
V(z)>0\quad\text{and}\quad V'(z)>0\quad \text{for all $z\geq\bar{z}$}.
\end{equation}
Furthermore, both $V$ and $V'$ are polynomially bounded, i.e.,
\begin{equation}
\label{eq:V-polynomial}
\abs{V'(z)}\leq b_1 +b_2 \abs{z}^n\quad \text{and}\quad \abs{V(z)}\leq b_1 +b_2 \abs{z}^{n+1},
\end{equation}
for some positive constants $b_i$, $i=1,2$, and a positive integer $n$.
\end{lemma}

\begin{proposition}
\label{prop:bar-V}
If Assumptions \ref{assumption:coefficients}-\ref{assumption:h} hold, then $\bar{\Phi}\subseteq \Phi_V$.
\end{proposition}
\begin{proof}[Proof of Proposition \ref{prop:bar-V}]
For any given $Q\in\bar{\Phi}$, i.e., $Q\in\Phi_j$ for some $j$, we need to show that the controlled process $Z=\{Z(t)\}_{t\geq0}$ under $Q$ as well as function $V$ defined in \eqref{eq:V} satisfy conditions \eqref{eq:condition1}-\eqref{eq:condition3}.

Let $z_b=\bar{z}\vee j \vee x$ ($x$ is the initial level) and  $Z_b=\{Z_b(t)\}_{t\geq0}$ be the reflected process with lower barrier $z_b$ and any initial level $z\in[z_b,\infty)$, then it follows from Remark 3.3 in \cite{HelmesETAL17} that $Z_b$ has a stationary distribution with density
\begin{align}
\label{eq:pi}
\pi(z)=0 \text{ for $z<z_b$}\quad\text{and}\quad
\pi(z)=\frac{\frac{1}{\sigma^2(z)}\exp(-\int_{z_b}^z \frac{2\mu(u)}{\sigma^2(u)}\,\mathrm{d}u)}
{\int_{z_b}^{\infty} \frac{1}{\sigma^2(z)}\exp(-\int_{z_b}^z \frac{2\mu(u)}{\sigma^2(u)}\,\mathrm{d}u)\,\mathrm{d}z}
\quad \text{for $z\geq z_b$}.
\end{align}
Note that the boundedness of $\mu$ and $\sigma$ in Assumption \ref{assumption:coefficients} implies that
\begin{equation}
\label{eq:int-f<infty}
\int_{\bar{z}}^{\infty}f(z)\pi(z)\,\mathrm{d}z<\infty\quad \text{for any polynomially bounded function $f$}.
\end{equation}
Denote $\bar{Z}_b=\{\bar{Z}_b(t)\}_{t\geq0}$ as the reflected process with lower barrier $z_b$ and a initial level given by a random variable with distribution \eqref{eq:pi}. Then, for any $t\geq 0$, $\bar{Z}_b(t)$ has the same distribution with density \eqref{eq:pi}.
We next show
\begin{equation}
\label{eq:Z<bar-Zb}
Z(t)\leq \bar{Z}_b(t)\quad \text{a.s. for any $t\geq0$}.
\end{equation}
At time zero, it follows from $z_b=\bar{z}\vee j \vee x$ that $Z(0-)=x\leq z_b\leq \bar{Z}_b(0)$ a.s.. Also, at any ordering time $\tau$ of $Z$, $z_b\geq j$ and $Q\in\Phi_j$ imply that $Z(\tau)\leq j\leq z_b\leq \bar{Z}_b(\tau)$ a.s.. Furthermore, during any two successive ordering times, the process $Z$ cannot move above $\bar{Z}_b$ through diffusion on each sample path since once $Z$ and $\bar{Z}_b$ become same at certain time, they will keep same thereafter until the next ordering time. Thus, \eqref{eq:Z<bar-Zb} holds.

We first prove \eqref{eq:condition1}. In fact, we have
\begin{align*}
\mathbb{E}_x\Big[\int_0^t \big(V'(Z(s))\sigma(Z(s))\big)^2\,\mathrm{d}s\Big]
&\leq \bar{\sigma}^2 \mathbb{E}_x\Big[\int_0^t\big(V'(Z(s))\big)^2\big(1_{\{Z(s)<\underline{s}\}}+1_{\{\underline{s}\leq Z(s)<\bar{z}\}}+1_{\{Z(s)\geq \bar{z}\}}\big)\,\mathrm{d}s \Big].
\end{align*}
It follows from \eqref{eq:V} that the first two terms are finite. For the last term, we have
\begin{align*}
\mathbb{E}_x\Big[\int_0^t\big(V'(Z(s))\big)^21_{\{Z(s)\geq \bar{z}\}}\,\mathrm{d}s\Big]
&\leq \mathbb{E}_x\Big[\int_0^t\big(b_1+b_2\abs{Z(s)}^n\big)^21_{\{Z(s)\geq \bar{z}\}}\,\mathrm{d}s\Big]\\
&\leq \mathbb{E}_x\Big[\int_0^t\big(b_1+b_2\bar{Z}_b(s)^n\big)^2\,\mathrm{d}s\Big]\\
&= t\cdot\int_{z_b}^{\infty}\big(b_1+b_2z^n\big)^2\pi(z)\,\mathrm{d}z\\
&<\infty,
\end{align*}
where the first inequality is from \eqref{eq:bar-z}-\eqref{eq:V-polynomial}, the second inequality is from \eqref{eq:Z<bar-Zb} and $\bar{Z}_b(t)\geq z_b\geq \bar{z}$ a.s., and the equality holds because $\bar{Z}_b(t)$ has the same distribution with density \eqref{eq:pi} for any $t\geq0$.
Therefore, we have proven \eqref{eq:condition1}.

We next prove \eqref{eq:condition3}. We have
\begin{align}
\mathbb{E}_x\big[\abs{V(Z(t))}1_{\{Z(t)\geq0\}}\big]
&=\mathbb{E}_x\big[\abs{V(Z(t))}1_{\{0\leq Z(t)<\bar{z}\}}\big]
+\mathbb{E}_x\big[V(Z(t))1_{\{Z(t)\geq\bar{z}\}}\big]\nonumber\\
&\leq \max_{z\in[0,\bar{z}]}V(z)
+\mathbb{E}_x\big[V(\bar{Z}_b(t))\big]
\label{eq:EV<constant}
\end{align}
where the inequality holds due to \eqref{eq:bar-z} and $\bar{Z}_b(t)\geq z_b\geq \bar{z}$ a.s..
Note that $\mathbb{E}_x\big[V(\bar{Z}_b(t))\big]=\int_{z_b}^{\infty} V(z)\pi(z)\,\mathrm{d}z$,
thus it follows from  \eqref{eq:V-polynomial} and \eqref{eq:int-f<infty} that the right side of \eqref{eq:EV<constant} is a finite constant and independent of $t$. Thus we obtain \eqref{eq:condition3}.

Finally, we prove \eqref{eq:condition2}. We first notice
\begin{align*}
\abs{V(Z(t))}&=\abs{V(Z(t))}\big(1_{\{Z(t)<\underline{s}\}}+1_{\{\underline{s}\leq Z(t)< 0\}}+1_{\{Z(t)\geq 0\}}\big).
\end{align*}
The definition of $V$ in \eqref{eq:V} implies that the first two terms are finite, and \eqref{eq:EV<constant} implies that the last term is also finite.
Thus,  \eqref{eq:condition2} holds.
\end{proof}

\section{Proof of Theorem \ref{thm:mainresults}}
\label{sec:comparison}

We, in this section, will prove that the $(s^{\star},S^{\star})$ policy is optimal among all admissible policies (i.e., Theorem \ref{thm:mainresults}) by a comparison theorem. Specifically,  for any admissible policy $Q\in\Phi$, if we can find a sequence $\{Q_j\in\Phi(j)\subseteq\bar{\Phi}:j=1,2,\cdots\}$ satisfying
\begin{equation}
\label{eq:lim-AC}
\limsup_{j\to\infty}\mathcal{C}(x,Q_j)\leq \mathcal{C}(x,Q)\quad \text{for $x\in\mathbb{R}$,}
\end{equation}
then the optimal policy in $\bar{\Phi}=\cup_{j=1}^{\infty} \Phi(j)$ must be optimal in $\Phi$.
From Propositions \ref{prop:optimal} and \ref{prop:bar-V}, we have proven that the $(s^{\star},S^{\star})$ policy defined in \eqref{eq:sS-star} is the best one in $\bar{\Phi}$. To eventually establish the global optimality of the $(s^{\star},S^{\star})$ policy in $\Phi$, what remains is to construct a sequence of $\{Q_j\in\Phi(j):j=1,2,\cdots\}$ for each $Q\in\Phi$ and prove \eqref{eq:lim-AC}.

For any given admissible policy $Q\in \Phi(j)$ (with $Z$ as the controlled inventory process under policy $Q$), the construction of the sequence of policies $\{Q_j\in\Phi(j):j=1,2,\cdots\}$ is same as that in \cite{HeETAL17}. However, a more general argument is required to tackle the technical issues arising from the general diffusion demand process.
Let $Q_j(t)$ denote the total order amount of policy $Q_j$ in $[0,t]$, and $Z_j=\{Z_j(t):t\geq0\}$ be the resulting inventory process under $Y_j$, i.e.,
\begin{equation}
\label{eq:Zm}
Z_j(t)=x-D_j(t)+Q_j(t),\quad t\geq0,
\end{equation}
where $D_j(t)=\int_0^t \mu(Z_j(s))\,\mathrm{d}s+\int_0^t \sigma(Z_j(s))\,\mathrm{d}B(s)$.
We define the jumps of $Q_j$ as follows; see  \cite{HeETAL17}.
 \begin{itemize}
 \item[$(\mathcal{J}1)$]  $\Delta Q_j(t)=0$ for $t$ satisfying $\Delta Q(t)> 0$ and $Z_j(t-)>j/2$;
 \item[$(\mathcal{J}2)$]  $\Delta Q_j(t)=\Delta Q(t)$ for $t$ satisfying $\Delta Q(t)> 0$, $Z_j(t-)\leq j/2$, and $Z_j(t-)+\Delta Q(t)\leq j$;
 \item[$(\mathcal{J}3)$]  $\Delta Q_j(t)=j-Z_j(t-)$ for $t$ satisfying $\Delta Q(t)> 0$, $Z_j(t-)\leq j/2$, and $Z_j(t-)+\Delta Q(t)> j$;
 \item[$(\mathcal{J}4)$]  $\Delta Q_j(t)=\max(\min(Z(t), j),0)$ for $t$ satisfying $Z_j(t-)= 0$.
 \end{itemize}

 \begin{proposition}[Comparison Theorem]
 \label{prop:comparison}
 Suppose Assumption \ref{assumption:coefficients}-\ref{assumption:h} hold. For any admissible policy $Q\in\Phi$, the policy sequence $\{Q_j\in\Phi_j:j=1,2,\cdots\}$ constructed by $(\mathcal{J}1)$-$(\mathcal{J}4)$ satisfies \eqref{eq:lim-AC}.
 \end{proposition}

\begin{proof}[Proof of Proposition \ref{prop:comparison}]
To prove \eqref{eq:lim-AC}, we need to compare the holding/shortage cost and ordering cost under $Q$ and $\{Q_j\in\Phi(j):j=1,2,\cdots\}$.

Consider the holding/shortage cost. It follows from the construction of $Q_j$ by $(\mathcal{J}1)$-$(\mathcal{J}4)$, we can easily have that on each sample path,
\begin{align}
\label{eq:Zj<Z}
Z_j(t)\leq Z(t) \text{ if $Z_j(t)\geq 0$}\quad \text{and}\quad Z_j(t)= Z(t)  \text{ if $Z_j(t)< 0$.}
\end{align}
By \eqref{eq:Zj<Z} and the properties of holding/shortage cost function $h$ in Assumption \ref{assumption:h}, we have that the holding/shortage cost incurred under $Q_j$ is no greater than that under $Q$.


Consider the ordering cost. We first show some properties of function $c$.
Since $c$ is a subadditive function in $\mathbb{R}_+$,  the limit $\lim_{\xi\to\infty}c(\xi)/\xi$ must exist and $\lim_{\xi\to\infty}c(\xi)/\xi=\inf_{\xi>0}c(\xi)/\xi$  (cf. Theorem 16.2.9 in \cite{Kuczma09book}). Let
\[
k:=\inf_{\xi>0}\frac{c(\xi)}{\xi}\quad \text{and}\quad K(\xi):=c(\xi)-k\xi.
\]
Then we have
\begin{equation}
\label{eq:K}
K(\xi)\geq0\quad \text{and}\quad \lim_{\xi\to\infty}\frac{K(\xi)}{\xi}=0.
\end{equation}
Thus, $k$ can be treated as the proportional cost and $K(\xi)$ as the setup cost for an order with quantity $\xi$, and the cumulative ordering cost up to time $t$ under $Q$ can be rewritten as
$\sum_{n=0}^{N(t)}K(\xi_n)+k Q(t)$.

We next consider the proportional cost.
The cumulative proportional costs up to time $t$ under $Q$ and $Q_j$ are
$kQ(t)$ and $kQ_j(t)$, respectively. We claim that for any $j=1,2,\cdots$,
\begin{equation}
\label{prop cost}
\limsup_{t\to\infty}\mathbb{E}_x[Q(t)]/t\geq \limsup_{t\to\infty}\mathbb{E}_x[Q_j(t)]/t.
\end{equation}
Suppose \eqref{prop cost} does not hold, i.e.,
\begin{equation}
\label{eq:contradiction}
a:=\limsup_{t\to\infty}\mathbb{E}_x[Q(t)]/t< \limsup_{t\to\infty}\mathbb{E}_x[Q_j(t)]/t:=b,
\end{equation}
which, implies that we can find a subsequence of ordering times $\{\theta_n\}_{n \geq 1}$ satisfying
\begin{equation}
\label{eq:b-big}
\lim_{n\to\infty}\mathbb{E}_x[Q_j(\theta_n)]/\theta_n=b.
\end{equation}
For this subsequence, we have
\begin{equation}
\label{eq:a-big}
\limsup_{n\to\infty}\mathbb{E}_x[Q(\theta_n)]/\theta_n\leq \limsup_{t\to\infty}\mathbb{E}_x[Q(t)]/t=a.
\end{equation}
Thus, it follows from \eqref{eq:contradiction}-\eqref{eq:a-big} that there must exists a $\bar{n}$ such that
\begin{equation}
\label{eq:EQ}
\mathbb{E}_x[Q(\theta_n)]< \mathbb{E}_x[Q_j(\theta_n)]\quad \text{for all $n\geq\bar{n}$}.
\end{equation}
Moreover, from \eqref{eq:Zj<Z} and the fact that $\mu(\cdot)$ is non-decreasing (see Assumption \ref{assumption:coefficients}(a)), we have
\begin{equation}
\label{eq:demand_compar}
\mathbb{E}_x[D(t)]\geq \mathbb{E}_x[D_j(t)], \quad \text{for all $t\geq0$},
\end{equation}
where $D_j(t)=\int_0^t \mu(Z_j(s))\,\mathrm{d}s+\int_0^t \sigma(Z_j(s))\,\mathrm{d}B(s)$.
Furthermore, it follows from \eqref{eq:Z} and \eqref{eq:Zm} that
$Z(t)=x-D(t)+Q(t)$ and $Z_j(t)=x-D_j(t)+Q_j(t)$,
which, together with \eqref{eq:EQ}-\eqref{eq:demand_compar}, imply that
\[
\mathbb{E}_x[Z(\theta_n)]<\mathbb{E}_x[Z_j(\theta_n)]\quad \text{for all $n\geq\bar{n}$},
\]
contradicting with \eqref{eq:Zj<Z}.
Therefore, \eqref{prop cost} holds.

It remains to consider the setup cost.
For function $K(\cdot)$, we can further claim
\begin{equation}
\label{eq:lim-sup-K}
\lim_{j\to\infty}\frac{\sup_{\xi\in[0,j]}K(\xi)}{j}=0.
\end{equation}
In fact, it follows from the second part in \eqref{eq:K} that for any $\epsilon>0$, there is a $n_{\epsilon}$ such that  $K(\xi)/\xi<\epsilon$ for all $\xi\geq n_{\epsilon}$. Further, there exits an $j_{\epsilon}\geq n_{\epsilon}$ such that $\sup_{\xi\in[0,n_{\epsilon}]}K(\xi)/j<\epsilon$ for all $j\geq j_{\epsilon}$.
Therefore, for all $j\geq j_{\epsilon}$,
\[
\frac{\sup_{\xi\in[0,j]}K(\xi)}{j}= \max\Big\{\frac{\sup_{\xi\in[0,n_{\epsilon}]}K(\xi)}{j},\frac{\sup_{\xi\in[n_{\epsilon},j]}K(\xi)}{j}\Big\}<\epsilon.
\]
Since $\epsilon$ is arbitrary, \eqref{eq:lim-sup-K} holds.

Now we consider the setup cost incurred by the orders under policy $Q_j$ in $(\mathcal{J}2)-(\mathcal{J}4)$. For the order in $(\mathcal{J}2)$, $Q_j$ and $Q$ incur the same setup cost.

Consider the orders under policy $Q_j$ in $(\mathcal{J}3)$. Let $t_1$ and $t_2$ denote any two consecutive ordering times with $t_1<t_2$. Let $X_j(t)=x-\int_0^t \mu(Z_j(u))\,\mathrm{d}u+\int_0^t \sigma(Z_j(t))\,\mathrm{d}B(u)$.
Recall the definition of $Z_j$ in \eqref{eq:Zm}, we have
\begin{align*}
Z_j(t_1)=X_j(t_1)+Q_{j}(t_1)\quad\text{and}\quad Z_j(t_2-)=X_j(t_2)+Q_{j}(t_2-),
\end{align*}
which, together with $Q_{j}(t_1)\leq Q_{j}(t_2-)$, imply
\[
X_j(t_1)-X_j(t_2)\geq Z_j(t_1)-Z_j(t_2-)=j-Z_j(t_2-)=Z_j(t_2)-Z_j(t_2-)=\Delta Q_j(t_2)\geq \frac{j}{2},
\]
where the first two equalities follow from $Z_j(t_1)=Z_j(t_2)=j$.
Let $\tau=\inf\{s\in(0,t_2-t_1]: X_j(t_1+s)=X_j(t_1)-j/2=j/2\}$. It follows from the second part in \eqref{eq:E-h-tau} and Assumption \ref{assumption:coefficients} that
\begin{align}
\mathbb{E}_x[\tau]
&=2\int_{\frac{j}{2}}^j \int_u^{\infty} \mathcal{M}(\mathrm{d}v)\,\mathrm{d}\mathcal{S}(u)\nonumber\\
&=2\int_{\frac{j}{2}}^j \int_u^{\infty}\frac{1}{\sigma^2(v)}\exp\Big(-\int_u^v \frac{2\mu(z)}{\sigma^2(z)}\,\mathrm{d}z\Big)\,\mathrm{d}v\,\mathrm{d}u\nonumber\\
&\geq \frac{2}{\bar{\sigma}^2}\int_{\frac{j}{2}}^j \int_u^{\infty}\exp\Big(-\frac{2\bar{\mu}}{\underline{\sigma}^2}(v-u)\Big)\,\mathrm{d}v\,\mathrm{d}u\nonumber\\
&=\frac{\underline{\sigma}^2 j}{2\bar{\mu}\bar{\sigma}^2}.\label{eq:E-tau}
\end{align}
Let $N_{j,1}(t)$ be the number of ordering in $(\mathcal{J}3)$ under $Q_j$ up to time $t$. Since $t_2-t_1\geq \tau$, we have
\[
\mathbb{E}_x[N_{j,1}(t)]\leq \frac{1}{\mathbb{E}_x[\tau]}t+1=\frac{2\bar{\mu}\bar{\sigma}^2}{\underline{\sigma}^2 j}t+1.
\]

Now consider the orders under $Q_j$ in $(\mathcal{J}4)$. Let $\tilde{t}_1$ and $\tilde{t}_2$ denote any two consecutive ordering times with $\tilde{t}_1<\tilde{t}_2$.
In this case, we claim that there must exist some $\tilde{t}_3\in [\tilde{t}_1,\tilde{t}_2)$ satisfying $Z_j(\tilde{t}_3)>j/2$. If $Z_j(\tilde{t}_1)\neq Z(\tilde{t}_1)$,
we must have $Z_j(\tilde{t}_1)=j$ and then choose $\tilde{t}_3=\tilde{t}_1$.
If $Z_j(\tilde{t}_1)=Z(\tilde{t}_1)$, assume that such $\tilde{t}_3$ does not exist in $[\tilde{t}_1,\tilde{t}_2)$, then the cases in $(\mathcal{J}1)$,  $(\mathcal{J}3)$, and  $(\mathcal{J}4)$ can not happen in $(\tilde{t}_1,\tilde{t}_2)$. This implies $Z_j(\tilde{t}_2-)= Z(\tilde{t}_2-)$, contradicting with the fact $Z_j(\tilde{t}_2-)\neq Z(\tilde{t}_2-)$.
Let $\tilde{\tau}=\inf\{s\in(0,\tilde{t}_2-\tilde{t}_3]:X_j(\tilde{t}_3+s)=X_j(\tilde{t}_3)-\frac{j}{2}\}$. Using the same derivations as in \eqref{eq:E-tau}, we have
\[
\mathbb{E}_x[\tilde{\tau}]\geq \frac{\underline{\sigma}^2 j}{2\bar{\mu}\bar{\sigma}^2}.
\]
Let $N_{j,2}(t)$ be the number of ordering in $(\mathcal{J}4)$ under $Q_j$ in $[0,t]$. Since $\tilde{t}_2-\tilde{t}_1\geq \tilde{t}_2-\tilde{t}_3\geq \tilde{\tau}$, we have
\[
\mathbb{E}_x[N_{j,2}(t)]\leq \frac{2\bar{\mu}\bar{\sigma}^2}{\underline{\sigma}^2 j}t+1.
\]

To sum up the holding/shortage cost, proportional cost, and setup cost discussed above, we have
\begin{align*}
\mathcal{C}(x,Q_j)-\mathcal{C}(x,Q)&\leq
\limsup_{t\to\infty}\frac{1}{t}\mathbb{E}_x\big[\mathbb{E}_x[N_{j,1}(t)]+\mathbb{E}_x[N_{j,2}(t)]\big]\sup_{\xi\in[0,j]}K(\xi)\\
&\leq \frac{4\bar{\mu}\bar{\sigma}^2}{\underline{\sigma}^2 }\frac{\sup_{\xi\in[0,j]}K(\xi)}{j},
\end{align*}
which, together with \eqref{eq:lim-sup-K}, implies that \eqref{eq:lim-AC}.
\end{proof}


\section{Concluding Remarks}
\label{sec:concluding}

In this paper, we used a two-step approach to prove the global optimality of an $(s,S)$ policy in an ergodic inventory control problem with inventory-dependent diffusion demand and general ordering costs. Specifically, we first applied a lower bound theorem to show the optimality of the selected policy in a subset of admissible policies, and then used a comparison theorem to establish the global optimality among all admissible policies.


\appendix

\section{Proof of Lemma \ref{lem:sS}}
\label{app:lemma-sS}

Let
\begin{equation}
\label{eq:gamma}
\gamma(s,S)=\frac{2\int_s^S \int_x^{\infty} h(y)\,\mathcal{M}(\mathrm{d}y)\,\mathrm{d}\mathcal{S}(x)}{2\int_s^S\int_{x}^{\infty} \,\mathcal{M}(\mathrm{d}y)\,\mathrm{d}\mathcal{S}(x)}.
\end{equation}
It follows from Assumptions \ref{assumption:coefficients} and \ref{assumption:h} (as well as Remark \ref{rem:h} (a)) that the conditions in
Lemma 2.1 in \cite{HelmesETAL18} hold.  Then, we have
\begin{align}
\lim_{s\to-\infty}\gamma(s,S)&=\lim_{(s,S)\to(-\infty,-\infty)}\gamma(s,S)=\lim_{z\to-\infty}h(z)=\infty\quad \text{and}\label{eq:lim-gamma}\\
\lim_{S\to\infty}\gamma(s,S)&=\lim_{(s,S)\to(\infty,\infty)}\gamma(s,S)=\lim_{z\to\infty}h(z)=\infty,\nonumber
\end{align}
which, together with the non-negativity of $c$ in Assumption \ref{assumption:c}, imply
\begin{align*}
\lim_{s\to-\infty}\alpha(s,S)&=\lim_{(s,S)\to(-\infty,-\infty)}\alpha(s,S)=\infty\quad \text{and}\quad \lim_{S\to\infty}\alpha(s,S)=\lim_{(s,S)\to(\infty,\infty)}\alpha(s,S)=\infty.
\end{align*}
Thus, we can find a finite positive number $B_1$ satisfying
\begin{equation}
\label{eq:alpha-1}
\inf_{s<S}\alpha(s,S) = \inf_{-B_1\leq s<S\leq B_1}\alpha(s,S).
\end{equation}
Let $\Delta=S-s$, then $\alpha(s,S)$ can be rewritten as
\[
\eta(s,\Delta)=\frac{2\int_s^{s+\Delta} \int_x^{\infty} h(y)\,\mathcal{M}(\mathrm{d}y)\,\mathrm{d}\mathcal{S}(x)+c(\Delta)}{2\int_s^{s+\Delta}\int_{x}^{\infty} \,\mathcal{M}(\mathrm{d}y)\,\mathrm{d}\mathcal{S}(x)}.
\]
From $\lim_{\Delta\downarrow 0} c(\Delta)>0$ (see Assumption \ref{assumption:c}), we have
$\lim_{\Delta\downarrow 0}\eta(s,\Delta)=\infty$,
thus we can find a finite positive number $B_2$ such that \eqref{eq:alpha-1} becomes
\[
\inf_{s<S}\alpha(s,S) =\min_{-B_1\leq s\leq B_1, B_2\leq \Delta\leq 2B_1}\eta(s,\Delta).
\]
Since $\eta(s,\Delta)$ is continuous in $s$, there exists an $s(\Delta)\in[-B_1,B_1]$ such that
\[
\eta(s(\Delta),\Delta)=\min_{-B_1\leq s\leq B_1}\eta(s,\Delta) \quad \text{for each $\Delta\in[B_2,2B_1]$}.
\]
Further, since $c(\Delta)$ is low semicontinuous and other parts in $\eta(s(\Delta),\Delta)$ is continuous in $\Delta$,
by the extreme value theorem (see Theorem B.2 in \cite{Puterman1994}), there exists a $\Delta^{\star}\in[B_2,2B_1]$ such that
\[
\eta(s(\Delta^{\star}),\Delta^{\star})\leq \eta(s(\Delta),\Delta)\quad \text{for all $\Delta\in[B_2,2B_1]$}.
\]
Let $s^{\star}=s(\Delta^{\star})$ and $S^{\star}=s^{\star}+\Delta^{\star}$, then we complete the proof.
\qed

\section{Proof of Lemma \ref{lem:underline-s}}
We show the existence of $\underline{s}$ satisfying \eqref{eq:underline-s1} and \eqref{eq:underline-s2} as follows: First, in part ($a$), we show that there exists an $\underline{s}_1\in(-\infty,s^{\star}]$ such that \eqref{eq:underline-s1} holds for any $\underline{s}\in(-\infty,\underline{s}_1]$; and in part ($b$), we show that we can find an $\underline{s}_2\in(-\infty,s^{\star}]$ such that $g'(z)-\alpha^{\star}\ell'(z)<0$ for any $z\leq \underline{s}_2$. Then, we let $\underline{s}=\underline{s}_1\wedge \underline{s}_2$, then both \eqref{eq:underline-s1} and \eqref{eq:underline-s2} hold.

($a$) First, by Assumption \ref{assumption:coefficients}, we have
\begin{align*}
\lim_{z\to-\infty}\int_{z}^{\infty}\,\mathcal{M}(\mathrm{d}u)
&=\lim_{z\to-\infty}\int_{z}^{\infty} \frac{1}{\sigma^2(u)}\exp\Big(-\int_c^u \frac{2\mu(y)}{\sigma^2(y)}\,\mathrm{d}y\Big)\, \mathrm{d}u\\
&\geq \lim_{z\to-\infty}\int_{z}^{c} \frac{1}{\sigma^2(u)}\exp\Big(-\int_c^u \frac{2\mu(y)}{\sigma^2(y)}\,\mathrm{d}y\Big)\, \mathrm{d}u\\
&\geq \lim_{z\to-\infty}\int_{z}^{c}\frac{1}{\bar{\sigma}^2}\exp\Big(\frac{2\underline{\mu}}{\bar{\sigma}^2}(c-u)\Big)\, \mathrm{d}u\\
&=\infty.
\end{align*}
Similarly, we have
\[
\lim_{z\to-\infty}\int_{z}^{\infty}h(u)\,\mathcal{M}(\mathrm{d}u)=\infty.
\]
Therefore, by L' H\^{o}pital's rule, we have
\[
\lim_{z\to -\infty}\frac{g(z)}{\ell(z)}
=\lim_{z\to -\infty}\frac{\int_{z}^{\infty}h(u)\,\mathcal{M}(\mathrm{d}u)}{\int_{z}^{\infty}\,\mathcal{M}(\mathrm{d}u)}
=\lim_{z\to-\infty}h(y)=\infty,
\]
which yields that we can find an $s^{\dagger}$ with $s^{\dagger}\leq s^{\star}$ such that
\begin{equation}
\label{eq:g/l>alpha1}
\frac{g(y)}{\ell(y)}\geq \alpha^{\star} \quad \text{for any $y\in(-\infty,s^{\dagger}]$}.
\end{equation}
Also, by L' H\^{o}pital's rule, we have
\begin{equation}
\label{eq:lim-g/l}
\lim_{z\to \infty}\frac{g(z)}{\ell(z)}
=\lim_{z\to \infty}\frac{\int_{z}^{\infty}h(u)\,\mathcal{M}(\mathrm{d}u)}{\int_{z}^{\infty}\,\mathcal{M}(\mathrm{d}u)}
=\lim_{y\to\infty}h(y)=\infty,
\end{equation}
which yields that there exists an $s^{\ddagger}$ with $s^{\ddagger}\geq s^{\star}$ such that
\begin{equation}
\label{eq:g/l>alpha2}
\frac{g(z)}{\ell(z)}\geq \alpha^{\star} \quad \text{for any $y\in[s^{\ddagger},\infty)$}.
\end{equation}
In addition, it follows from \eqref{eq:gamma} and \eqref{eq:lim-gamma} that
\[
\lim_{s\to-\infty}\frac{\int_{s}^S g(y)\,\mathrm{d}y}{\int_{s}^S \ell(y)\,\mathrm{d}y}=\lim_{y\to-\infty}h(y)=\infty
\quad \text{for any fixed $S\in\mathbb{R}$}.
\]
Then, there exists an $\underline{s}_1$ with $\underline{s}_1\leq s^{\dagger}$ such that
\begin{equation}
\label{eq:g/l>alpha3}
\frac{\int_{s}^S g(y)\,\mathrm{d}y}{\int_{s}^S \ell(y)\,\mathrm{d}y}\geq \alpha^{\star} \quad \text{for $S\in[s^{\dagger},s^{\ddagger}]$ and $s\leq \underline{s}_1$.}
\end{equation}

Now we can show that \eqref{eq:underline-s1} holds for any $\underline{s}\in(-\infty,\underline{s}_1]$.
If $s\geq \underline{s}$, we have
\[
\underline{\alpha}(s,S)=\alpha(s,S)\geq \alpha^{\star},
\]
where the inequality follows from \eqref{eq:sS-star}.
Next, we prove the case when $s<\underline{s}$ in three subcases: $S\leq s^{\dagger}$, $s^{\dagger}<S<s^{\ddagger}$, and $S\geq s^{\ddagger}$.
If $S\leq s^{\dagger}$, we have
\[
\underline{\alpha}(s,S)\geq \frac{\int_s^S g(y\vee\underline{s})\,\mathrm{d}y}{\int_s^S \ell(y\vee\underline{s})\,\mathrm{d}y}
\geq \frac{\int_s^S \alpha^{\star}\cdot\ell(y\vee\underline{s})\,\mathrm{d}y}{\int_s^S \ell(y\vee\underline{s})\,\mathrm{d}y} =\alpha^{\star},
\]
where the first inequlity follow the non-negativity of $c(\cdot)$ in Assumption \ref{assumption:c}, and the second inequality follows \eqref{eq:g/l>alpha1} with $s<S\leq s^{\dagger}$.
If $s^{\dagger}<S<s^{\ddagger}$, we have
\[
\underline{\alpha}(s,S)\geq \frac{\int_s^S g(y\vee\underline{s})\,\mathrm{d}y}{\int_s^S \ell(y\vee\underline{s})\,\mathrm{d}y}
=\frac{g(\underline{s})(\underline{s}-s)+\int_{\underline{s}}^S g(y\vee\underline{s})\,\mathrm{d}y}{\ell(\underline{s})(\underline{s}-s)+\int_{\underline{s}}^S \ell(y\vee\underline{s})\,\mathrm{d}y}
\geq \alpha^{\star},
\]
the the last inequality is derived from \eqref{eq:g/l>alpha1} and \eqref{eq:g/l>alpha3} with $\underline{s}\leq \underline{s}_1\leq s^{\dagger}<S<s^{\ddagger}$.
If $S\geq s^{\ddagger}$, we have
\[
\underline{\alpha}(s,S)\geq \frac{\int_s^S g(y\vee\underline{s})\,\mathrm{d}y}{\int_s^S \ell(y\vee\underline{s})\,\mathrm{d}y}
=\frac{g(\underline{s})(\underline{s}-s)+\int_{\underline{s}}^{s^{\ddagger}} g(y\vee\underline{s})\,\mathrm{d}y+\int_{s^{\ddagger}}^S g(y\vee\underline{s})\,\mathrm{d}y}{\ell(\underline{s})(\underline{s}-s)+\int_{\underline{s}}^{s^\ddagger} \ell(y\vee\underline{s})\,\mathrm{d}y+\int_{s^\ddagger}^{S} \ell(y\vee\underline{s})\,\mathrm{d}y}
\geq \alpha^{\star},
\]
where the last inequality holds due to \eqref{eq:g/l>alpha1}, \eqref{eq:g/l>alpha2}, and \eqref{eq:g/l>alpha3}.

($b$)  To prove that we can find an $\underline{s}_2\in(-\infty,s^{\star}]$  such that $g'(z)-\alpha^{\star}\ell'(z)<0$ for any $z\leq \underline{s}_2$, we will claim that
\[
\lim_{z\to-\infty} [g'(z)-\alpha^{\star}\ell'(z)]<0.
\]
It follows from the convexity of $h$ in Assumption \ref{assumption:h} that there exist $c_0>0$ and $z_0<0$ such that for all $z<z_0$,
\begin{equation}
\label{eq:z0}
h'(z)<-c_0.
\end{equation}
Then, for $z<z_0$, we rewrite $g'(z)-\alpha^{\star}\ell'(z)$ as
\begin{align}
&g'(z)-\alpha^{\star}\ell'(z)\nonumber\\
&=\frac{2\mu(z)}{\sigma^2(z)}\Big[g(z)-\alpha^{\star}\ell(z)-\frac{h(z)-\alpha^{\star}}{\mu(z)}\Big]\nonumber\\
&=\frac{2\mu(z)}{\sigma^2(z)}\int_0^{\infty}\Big[ \frac{2}{\sigma^2(y+z)}\exp\Big(-\int_z^{y+z}\frac{2\mu(u)}{\sigma^2(u)}\,\mathrm{d}u\Big)\Big(\big(h(y+z)-\alpha^{\star}\big)-\big(h(z)-\alpha^{\star}\big)\frac{\mu(y+z)}{\mu(z)}\Big)\Big]\,\mathrm{d}y\nonumber\\
&=\frac{2\mu(z)}{\sigma^2(z)}\big(\Lambda_1(z)-\Lambda_2(z)+\Lambda_3(z)\big),\nonumber
\end{align}
where the second equality holds because
\begin{equation*}
\int_0^{\infty} \frac{2\mu(y+z)}{\sigma^2(y+z)}\exp\Big(-\int_z^{y+z}\frac{2\mu(u)}{\sigma^2(u)}\,\mathrm{d}u\Big)\,\mathrm{d}y=1,
\end{equation*}
and in the last equality,
\begin{align*}
\Lambda_1(z)&=\int_{z_0-z}^{\infty}\Big[ \frac{2}{\sigma^2(y+z)}\exp\Big(-\int_z^{y+z}\frac{2\mu(u)}{\sigma^2(u)}\,\mathrm{d}u\Big)\big(h(y+z)-\alpha^{\star}\big)\Big]\,\mathrm{d}y,\\
\Lambda_2(z)&=\int_{z_0-z}^{\infty}\Big[ \frac{2}{\sigma^2(y+z)}\exp\Big(-\int_z^{y+z}\frac{2\mu(u)}{\sigma^2(u)}\,\mathrm{d}u\Big)\big(h(z)-\alpha^{\star}\big)\frac{\mu(y+z)}{\mu(z)}\Big]\,\mathrm{d}y,\quad \text{and}\\
\Lambda_3(z)&=\int_0^{z_0-z}\Big[ \frac{2}{\sigma^2(y+z)}\exp\Big(-\int_z^{y+z}\frac{2\mu(u)}{\sigma^2(u)}\,\mathrm{d}u\Big)\Big(\big(h(y+z)-\alpha^{\star}\big)-\big(h(z)-\alpha^{\star}\big)\frac{\mu(y+z)}{\mu(z)}\Big)\Big]\,\mathrm{d}y.
\end{align*}
If we can prove
\begin{equation}
\lim_{z\to-\infty}\Lambda_1(z)=\lim_{z\to-\infty}\Lambda_2(z)=0\quad \text{and} \quad \lim_{z\to-\infty}\Lambda_3(z)<0,\label{eq:lim-3}
\end{equation} then it follows from the positiveness of $\mu$ and the boundedness of $\mu$ and $\sigma$ (see Assumption \ref{assumption:coefficients}) that
\[
\lim_{z\to-\infty}  [g'(z)-\alpha^{\star}\ell'(z)]=\lim_{z\to-\infty}\frac{2\mu(z)}{\sigma^2(z)}\big(\Lambda_1(z)-\Lambda_2(z)+\Lambda_3(z)\big)<0.
\]
Thus, it remains to prove \eqref{eq:lim-3}.

First,  for $z<z_0$,  we rewrite $\Lambda_1$ as
\begin{align*}
\Lambda_1(z)
=\exp\Big(-\int_z^{z_0}\frac{2\mu(u)}{\sigma^2(u)}\,\mathrm{d}u\Big)
\int_{z_0}^{\infty}\Big[ \frac{2}{\sigma^2(y)}\exp\Big(-\int_{z_0}^{y}\frac{2\mu(u)}{\sigma^2(u)}\,\mathrm{d}u\Big)\big(h(y)-\alpha^{\star}\big)\Big]\,\mathrm{d}y.
\end{align*}
Since $\lim_{\abs{z}\to\infty} h(z)=\infty$ (Remark \ref{rem:h} ($b$)), there exist a $z_1>0$ such that for $\abs{z}>z_1$
\begin{equation}
\label{eq:h>gamma}
h(z)\geq \alpha^{\star},
\end{equation}
which, together with the polynomial boundedness of $h$ (Assumption \ref{assumption:h}) , implies
\begin{align*}
\int_{z_1}^{\infty}\Big[ \frac{2}{\sigma^2(y)}\exp\Big(-\int_{z_0}^{y}\frac{2\mu(u)}{\sigma^2(u)}\,\mathrm{d}u\Big)\big(h(y)-\alpha^{\star}\big)\Big]\,\mathrm{d}y&>0\quad\text{and}\\
\int_{z_1}^{\infty}\Big[ \frac{2}{\sigma^2(y)}\exp\Big(-\int_{z_0}^{y}\frac{2\mu(u)}{\sigma^2(u)}\,\mathrm{d}u\Big)\big(h(y)-\alpha^{\star}\big)\Big]\,\mathrm{d}y
&\leq \int_{z_1}^{\infty}\Big[ \frac{2}{\underline{\sigma}^2}\exp\Big(-\frac{2\underline{\mu}}{\bar{\sigma}^2}(y-z_0)\Big)\big(h(y)-\alpha^{\star}\big)\Big]\,\mathrm{d}y\\
&<\infty.
\end{align*}
Thus,
\begin{align*}
&\int_{z_0}^{\infty}\Big[ \frac{2}{\sigma^2(y)}\exp\Big(-\int_{z_0}^{y}\frac{2\mu(u)}{\sigma^2(u)}\,\mathrm{d}u\Big)\big(h(y)-\alpha^{\star}\big)\Big]\,\mathrm{d}y\\
&\quad=\big(\int_{z_0}^{z_1}+\int_{z_1}^{\infty}\Big)\Big[ \frac{2}{\sigma^2(y)}\exp\Big(-\int_{z_0}^{y}\frac{2\mu(u)}{\sigma^2(u)}\,\mathrm{d}u\Big)\big(h(y)-\alpha^{\star}\big)\Big]\,\mathrm{d}y
\end{align*}
is a finite number.
Furthermore, the boundedness of $\mu$ and $\sigma$ in Assumption \ref{assumption:h} implies
\[
\lim_{z\to-\infty}\exp\Big(-\int_z^{z_0}\frac{2\mu(u)}{\sigma^2(u)}\,\mathrm{d}u\Big)=0.
\]
Therefore, we have
\[
\lim_{z\to-\infty}\Lambda_1(z)=\lim_{z\to-\infty}\exp\Big(-\int_z^{z_0}\frac{2\mu(u)}{\sigma^2(u)}\,\mathrm{d}u\Big)
\int_{z_0}^{\infty}\Big[ \frac{2}{\sigma^2(y)}\exp\Big(-\int_{z_0}^{y}\frac{2\mu(u)}{\sigma^2(u)}\,\mathrm{d}u\Big)\big(h(y)-\alpha^{\star}\big)\Big]\,\mathrm{d}y=0.
\]

Second, \eqref{eq:h>gamma} and the boundedness of $\mu$ and $\sigma$ imply that for $z<-z_1$,
\begin{align*}
\Lambda_2(z)
&=\frac{h(z)-\alpha^{\star}}{\mu(z)}\int_{z_0-z}^{\infty} \frac{2\mu(y+z)}{\sigma^2(y+z)}\exp\Big(-\int_z^{y+z}\frac{2\mu(u)}{\sigma^2(u)}\,\mathrm{d}u\Big)\,\mathrm{d}y \\
&\leq \frac{h(z)-\alpha^{\star}}{\mu(z)}\int_{z_0-z}^{\infty} \frac{2\bar{\mu}}{\underline{\sigma}^2}\exp\Big(-\frac{2\underline{\mu}}{\bar{\sigma}^2}y\Big)\,\mathrm{d}y \\
&=\frac{\bar{\mu}\bar{\sigma}^2}{\underline{\mu}\underline{\sigma}^2}\frac{h(z)-\alpha^{\star}}{\mu(z)}
\exp\Big(-\frac{2\underline{\mu}}{\bar{\sigma}^2}(z_0-z)\Big).
\end{align*}
Therefore,
\[
0\leq \lim_{z\to-\infty}\Lambda_2(z)\leq \lim_{z\to-\infty}\frac{\bar{\mu}\bar{\sigma}^2}{\underline{\mu}\underline{\sigma}^2}\frac{h(z)-\alpha^{\star}}{\mu(z)}
\exp\Big(-\frac{2\underline{\mu}}{\bar{\sigma}^2}(z_0-z)\Big)=0,
\]
where the equality follows from the polynomial boundedness of $h$. Thus, we have
\[
\lim_{z\to-\infty}\Lambda_2(z)=0.
\]

Finally, we have
\begin{align*}
\lim_{z\to-\infty}\Lambda_3(z)
&\leq
\lim_{z\to-\infty}\int_0^{z_0-z}\Big[ \frac{2}{\sigma^2(y+z)}\exp\Big(-\int_z^{y+z}\frac{2\mu(u)}{\sigma^2(u)}\,\mathrm{d}u\Big)\big(h(y+z)-h(z)\big)\Big]\,\mathrm{d}y\\
&\leq\lim_{z\to-\infty}\int_0^{z_0-z} \frac{2}{\underline{\sigma}^2} \exp\Big(-\frac{2\underline{\mu}}{\bar{\sigma}^2}y\Big)(-c_0) y\,\mathrm{d}y\\
&=-\frac{c_0\bar{\sigma}^4}{2\underline{\mu}^2\underline{\sigma}^2}\\
&<0,
\end{align*}
where the first inequality holds due to \eqref{eq:h>gamma} and that $\mu(\cdot)$ is non-decreasing (see Assumption \ref{assumption:coefficients}(a)),
the second inequality is derived from \eqref{eq:z0}.
\qed

\section{Proof of Lemma \ref{lem:V}}
To prove \eqref{eq:bar-z}, we only need to prove
\begin{align}
\label{eq:lim-V'>0}
\lim_{z\to\infty}V'(z)>0,
\end{align}
which yields $\lim_{z\to\infty}V(z)=\infty$, and then \eqref{eq:bar-z} holds. We next prove \eqref{eq:lim-V'>0}. First, we have
\begin{align*}
\lim_{z\to\infty}g(z)
&=\lim_{z\to\infty}2\int_z^{\infty}\frac{1}{\sigma^2(u)}h(u)\exp\Big(-\int_z^u\frac{2\mu(y)}{\sigma^2(y)}\,\mathrm{d}y\Big)\,\mathrm{d}u\\
&\geq \lim_{z\to\infty}\frac{2}{\bar{\sigma}^2}h(z) \int_z^{\infty}\exp\Big(-\frac{2\bar{\mu}}{\underline{\sigma}^2}(u-z)\Big)\,\mathrm{d}u\\
&=\lim_{z\to\infty}\frac{\underline{\sigma}^2}{\bar{\mu}\bar{\sigma}^2}h(z)\\
&=\infty,
\end{align*}
where the inequality follows from $h'(z)>0$ for $z>0$ (Assumption \ref{assumption:h}) and the boundedness of $\mu$ and $\sigma$ in Assumption \ref{assumption:coefficients}.
This, together with \eqref{eq:lim-g/l} and the definition of $V$ in \eqref{eq:V}, implies that
\[
\lim_{z\to\infty}V'(z)=\lim_{z\to\infty}[g(z)-\alpha^{\star}\ell(z)]=\infty.
\]

Finally, \eqref{eq:V-polynomial} can be implied by the polynomial boundedness of $h$.
\qed

\bibliographystyle{plain}
\bibliography{References}

\end{document}